\documentclass[journal,twoside,web]{ieeecolor}
\usepackage{generic}
\usepackage[utf8]{inputenc}
\usepackage{amsmath}
\usepackage{amsfonts}
\usepackage{amssymb}
\let\labelindent\relax \usepackage{enumitem}   
\usepackage{graphicx}   
\usepackage{color}
\usepackage{comment}
\usepackage{graphicx}
\usepackage{cite}
\newtheorem{theorem}{Theorem}   
\newtheorem{lemma}{Lemma}

\newtheorem{proposition}{Proposition}

\newtheorem{example}{Example}
\newtheorem{remark}{Remark}

\DeclareMathOperator{\rank}{rank}
\DeclareMathOperator{\im}{im}

\DeclareMathOperator{\leftker}{leftker}
\newcommand{\norm}[1]{\left\lVert#1\right\rVert}

\newcommand{\bmx}{\mathbf{x}}

\newcommand{\bmy}{\mathbf{y}}

\newcommand{\calH}{\mathcal{H}}                       

\title{Beyond Persistent Excitation: Online Experiment Design for Data-Driven Modeling and Control
}

\author{Henk J. van Waarde
\thanks{The author is with the Control Group, Department of Engineering, University of Cambridge, Trumpington Street, Cambridge CB2 1PZ, UK. Email: {\tt\small hv280@cam.ac.uk}.
}%
}

\begin{document}

\maketitle
\thispagestyle{empty}
\pagestyle{empty}

\begin{abstract}
This paper presents a new experiment design method for data-driven modeling and control. The idea is to select inputs \emph{online} (using past input/output data), leading to desirable rank properties of data Hankel matrices. In comparison to the classical persistency of excitation condition, this online approach requires less data samples and is even shown to be completely sample efficient.
\end{abstract}


\section{Introduction}

\IEEEPARstart{R}{ecently}, there has been an increasing interest in the direct design of controllers using data \cite{Markovsky2008,Maupong2017,Coulson2019,DePersis2020,
vanWaarde2020,Berberich2020,vanWaarde2020d,Coulson2020,Xue2020,
vanWaarde2020b,Guo2020,Dai2021}. Several contributions study how controllers can be obtained from a given batch of (informative) data, even in the presence of noise \cite{DePersis2020,Berberich2020,vanWaarde2020d} and for classes of nonlinear systems \cite{Guo2020,Dai2021}. The question of how to \emph{obtain} such informative data sets, however, is largely open. For data-driven control to become an end-to-end solution, there is a need for new \emph{experiment design} methods to empower the data-based design. This is true especially for settings including noise and nonlinear dynamics. However, even for linear systems with exact data, current experiment design methods are not sample efficient.  

In this paper we will explore a new idea for designing experiments for data-driven modeling and control. Experiment design is a classical problem that has been mostly studied in the parametric identification literature. An established idea is to optimize a measure of the expected accuracy of the parameter estimates subject to input power constraints \cite{Goodwin1977,Gevers1986}. This problem is usually tackled in the frequency domain and convex formulations have been provided in \cite{Jansson2005}. The dual problem of finding the ``least costly" input achieving a fixed level of parameter accuracy has also been studied \cite{Bombois2006}, in a closed-loop setting. 

Less results are known in the context of non-parametric methods. In this area, a state-of-the-art result for linear time-invariant systems is \emph{Willems et al.'s fundamental lemma} \cite{Willems2005}. The idea behind this method is to select an input that is persistently exciting, which implies that a Hankel matrix of measured inputs and outputs satisfies a rank condition. This rank property is important, since it guarantees that \emph{all} trajectories of the system can be parameterized in terms of the measured trajectory. Essentially, the Hankel matrix of measured inputs and outputs serves as a non-parametric model of the system. This idea is simple yet powerful, and has been successfully employed in a number of recent publications on data-driven simulation and control, see e.g. \cite{Markovsky2008,Maupong2017,
Coulson2019}.

Despite its elegance and clear impact on data-driven methods\-, the fundamental lemma has some limitations. 
\vspace{-2pt}
\begin{enumerate}
\item It is not \emph{sample efficient} in the sense that it uses more data than strictly necessary. This is because the lemma works with inputs that are persistently exciting of a certain order, which imposes a conservative lower bound on the number of samples. Large data sets are challenging from a computational point of view, especially for the real-time implementation of controllers, see e.g. \cite{Fabiani2020}.
\item It is not applicable to \emph{noisy data}. A possible approach to deal with this is to perform approximate data-driven simulation and control in a maximum likelihood framework \cite{Yin2020}, for which an experiment design method was developed in \cite{Ianelli2020b}. Another line of work establishes data-driven control design methods with guaranteed stability and performance in the presence of bounded process noise \cite{DePersis2020,Berberich2020,vanWaarde2020d}. However, for this, experiment design methods are missing and it is unclear whether persistency of excitation is helpful in this context.
\end{enumerate}
\vspace{-2pt}

The purpose of this paper is to introduce another angle of attack to experiment design. In contrast to Willems' fundamental lemma, we do not use persistently exciting inputs, but instead design the inputs \emph{online}. This means that at each time step, the input is selected on the basis of inputs and outputs that have been collected at previous time steps. Such an online approach is natural because the first samples of an experiment already contain valuable partial information about the system, which can be exploited in the design of the remaining samples. Online approaches appear in various contexts in systems and control. We mention iterative feedback tuning \cite{Hjalmarsson1998} that utilizes repeated (closed-loop) experiments, contributions to dual control \cite{Filatov2004,Marafioti2014,Larsson2016} that trade-off exploration and exploitation, and adaptive experiment design \cite{Lindqvist2001,Gerencser2017} where optimal experiment design is combined with adaptive parameter estimation. Nonetheless, in the context of Willems' lemma, online input design has not received any attention.

As our main contributions, we provide online experiment design methods for both input/state and input/output systems. In both settings, we formally prove that the methods are sample efficient. This completely resolves the problem in 1), and thus shows an advantage of online experiment design over the classical persistency of excitation condition. Analogous to the fundamental lemma, our technical results are presented for noise-free data.  The extension to noisy data is left for future work, and some ideas for this will be discussed in Section~\ref{sec:conc}.  

\subsection{Notation and terminology}

\noindent 
The \emph{left kernel} of a real matrix $M$ is denoted by $\leftker M$. Consider a signal $f: \mathbb{Z} \to \mathbb{R}^\bullet$ and let $i,j \in \mathbb{Z}$ be integers such that $i \leq j$. We denote by $f_{[i,j]}$ the restriction of $f$ to the interval $[i,j]$, that is, 
\begin{equation*}
    f_{[i,j]} := \begin{bmatrix} f(i)^\top & f(i+1)^\top & \cdots & f(j)^\top
    \end{bmatrix}^\top.
\end{equation*}
With slight abuse of notation, we will also use the notation $f_{[i,j]}$ to refer to the sequence $f(i),f(i+1),\dots,f(j)$. 
Let $k$ be a positive integer such that $k \leq j-i+1$ and define the \emph{Hankel matrix} of depth $k$, associated with $f_{[i,j]}$, as
\begin{equation*}
    \mathcal{H}_{k}(f_{[i,j]}) := 
    \begingroup 
    \setlength\arraycolsep{2pt}
    \begin{bmatrix} f(i) & f(i+1) & \cdots & f(j-k+1) \\
    f(i+1) & f(i+2) & \cdots & f(j-k+2) \\
    \vdots & \vdots & & \vdots \\
    f(i+k-1) & f(i+k) & \cdots & f(j) 
    \end{bmatrix}.
    \endgroup
\end{equation*}
Note that the subscript $k$ refers to the number of block rows of $\mathcal{H}_k$. The sequence $f_{[i,j]}$ is called \emph{persistently exciting of order $k$} 
if $\mathcal{H}_{k}(f_{[i,j]})$ has full row rank.

\section{Rank conditions on data matrices}
\label{sec:rank}

Consider the linear time-invariant (LTI) system
\begin{subequations}\label{system}
\begin{align} 
\mathbf{x}(t+1) &= A\mathbf{x}(t) + B\mathbf{u}(t) \label{systema} \\
\mathbf{y}(t) &= C\mathbf{x}(t) + D \mathbf{u}(t), \label{systemb}
\end{align}
\end{subequations}
where $\mathbf{x} \in \mathbb{R}^n$ denotes the state, $\mathbf{u} \in \mathbb{R}^m$ is the input and $\mathbf{y} \in \mathbb{R}^p$ is the output. Throughout the paper, we will assume that \eqref{system} is \emph{minimal}, i.e., the pair $(A,B)$ is controllable and $(C,A)$ is observable.  The \emph{lag} $\ell$ of \eqref{system} can be defined as the smallest integer $i$ for which the observability matrix 
\begin{equation}
\label{observabilitymatrix}
\mathcal{O}_i := \begin{bmatrix}
	C^\top & (CA)^\top & (CA^2)^\top & \cdots & (CA^{i-1})^\top \end{bmatrix}^\top
\end{equation}
has rank $n$. Now, let $(u_{[0,T-1]},y_{[0,T-1]})$ be an input/output trajectory of \eqref{system} of length $T \geq \ell$. In addition, let $L \geq \ell$ be an integer. The main goal of this article is to provide a new method to design the input sequence $u_{[0,T-1]}$ so that the resulting input/output Hankel matrix 
\begin{equation}
\label{inputoutputHankel}
    \begin{bmatrix}
\mathcal{H}_{L}(y_{[0,T-1]}) \\                
                \mathcal{H}_{L}(u_{[0,T-1]}) 
\end{bmatrix} = \begingroup 
    \setlength\arraycolsep{2pt}
    \begin{bmatrix} 
    y(0) & y(1) & \cdots & y(T-L) \\
    \vdots & \vdots & & \vdots \\
    y(L-1) & y(L) & \cdots & y(T-1) \\
    u(0) & u(1) & \cdots & u(T-L) \\
    \vdots & \vdots & & \vdots \\
    u(L-1) & u(L) & \cdots & u(T-1)   
    \end{bmatrix}
    \endgroup
\end{equation}
has rank $n+mL$. This rank condition plays a fundamental role in modeling and control using data. Indeed, it implies that any length-$L$ input/output trajectory of \eqref{system} can be obtained from \eqref{inputoutputHankel}. More specifically, if \eqref{inputoutputHankel} has rank $n+mL$ then $(\bar{u}_{[0,L-1]},\bar{y}_{[0,L-1]})$ is an input/output trajectory of \eqref{system} if and only if 
\begin{equation*}
    \begin{bmatrix}
      \bar{y}_{[0,L-1]} \\  \bar{u}_{[0,L-1]} 
    \end{bmatrix} \in \im \begin{bmatrix}
				\mathcal{H}_{L}(y_{[0,T-1]}) \\                
                \mathcal{H}_{L}(u_{[0,T-1]}) 
    \end{bmatrix}.
\end{equation*}
This parameterization of input/output trajectories has been used extensively to simulate and control dynamical systems, see e.g., \cite{Markovsky2008,Maupong2017,Coulson2019}. If $L > \ell$, the rank condition on \eqref{inputoutputHankel} even implies that \emph{all} input/output trajectories of \eqref{system} (not just those of length $L$) can be obtained from $(u_{[0,T-1]},y_{[0,T-1]})$, see \cite{Markovsky2005b}. In this case, the input/output behavior of \eqref{system} is \emph{identifiable} \cite{Markovsky2020} and $(A,B,C,D)$ can be computed up to similarity transformation, e.g., using subspace methods \cite{Verhaegen2007}.

A notable special case is that of \emph{full state measurement}, i.e., $\bmy(t) = \bmx(t)$ and $L = \ell = 1$. In this case, \eqref{inputoutputHankel} reduces to 
\begin{equation}
\label{ismatrix}
\begin{bmatrix}
\calH_1(x_{[0,T-1]}) \\ \calH_1(u_{[0,T-1]})
\end{bmatrix} = \begin{bmatrix}
x(0) & x(1) & \cdots & x(T-1) \\
u(0) & u(1) & \cdots & u(T-1)
\end{bmatrix}.
\end{equation}
If \eqref{ismatrix} has rank $n+m$, the input/state trajectory $(u_{[0,T-1]},x_{[0,T]})$ fully captures the behavior of \eqref{systema} which allows the unique identification of $A$ and $B$.

\subsection{Recap of Willems et al.'s fundamental lemma}
\label{sectionWillemslemma}

Willems et al.'s fundamental lemma \cite{Willems2005} is an important \emph{experiment design} result. It reveals that the rank condition on \eqref{inputoutputHankel} is satisfied if the input sequence is chosen to be persistently exciting. In a state-space setting, this result can be stated as follows \cite{Willems2005,vanWaarde2020c}.

\begin{proposition}
\label{theoremWillems}
Consider the minimal system \eqref{system}. Let $(u_{[0,T-1]},y_{[0,T-1]})$ be an input/output trajectory of \eqref{system} and let $L \geq \ell$. If the input $u_{[0,T-1]}$ is persistently exciting of order $n+L$ then \eqref{inputoutputHankel} has rank $n+mL$.
\end{proposition}

We note that persistency of excitation of order $n+L$ requires at least $T \geq (m+1)(n+L)-1$ data samples.

\section{Online input design}
\label{sec:online}
The main contribution of the paper will be to provide a new input design technique to guarantee the rank property of \eqref{inputoutputHankel}. In contrast to Willems' fundamental lemma, this method will not rely on applying inputs that are persistently exciting of order $n+L$. Rather, each input $u(t)$ is selected \emph{online}, by making use of the previous samples $y(0),y(1),\dots,y(t-1)$ and $u(0),u(1),\dots,u(t-1)$, see Figure~\ref{fig:online}. 

\begin{figure}[h!]
\centering
\includegraphics[width=0.34\textwidth]{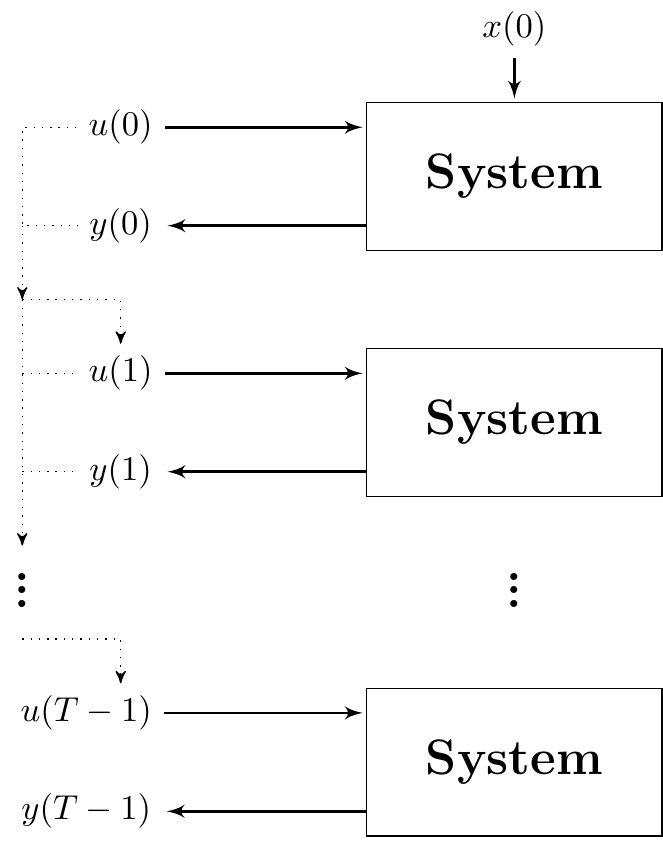}
\caption{Schematic representation of the online input design approach.}
\label{fig:online}
\end{figure}

\subsection{Input design for input/state systems}

To explain the idea, we will start with the simplest setting of \emph{input/state} data collected from system \eqref{systema}. In this setting, the purpose is to design a sequence of inputs $u(0),u(1),\dots,u(T-1)$ so that the matrix \eqref{ismatrix} has full row rank $n+m$. As mentioned, we opt for the \emph{online} design of the inputs. Specifically, this means that for the design of $u(t)$ we make use of the collected states\footnote{In this setting, we use $x(t)$ in the design since it is independent of $u(t)$.} $x(0),x(1),\dots,x(t)$ and inputs $u(0),u(1),\dots,u(t-1)$. The following theorem shows how such inputs and a time horizon $T$ can be designed.

\begin{theorem}
	\label{thmissys}
	Consider the controllable system \eqref{systema}. Define $T := n+m$. Select a nonzero $u(0) \in \mathbb{R}^m$, and design the input $u(t)$ (for $t = 1,2,\dots,T-1$) as follows:
	\begin{itemize}
		\item If $x(t) \not\in \im \calH_1(x_{[0,t-1]})$, select $u(t) \in \mathbb{R}^m$ arbitrarily. 
		\item If $x(t) \in \im \calH_1(x_{[0,t-1]})$ there exists a $\xi \in \mathbb{R}^n$ and a nonzero $\eta \in \mathbb{R}^m$ such that
		\begin{equation}
		\label{xieta}
		\begin{bmatrix} \xi^\top & \eta^\top \end{bmatrix} \begin{bmatrix}
		\calH_1(x_{[0,t-1]}) \\ \calH_1(u_{[0,t-1]})
		\end{bmatrix} = 0. 
		\end{equation}
		In this case, select $u(t)$ such that $\xi^\top x(t) + \eta^\top u(t) \neq 0$. 		
	\end{itemize}
	Then, we have that
	\begin{equation}
	\label{fullrank}
	\rank \begin{bmatrix}
	\calH_1(x_{[0,T-1]}) \\ \calH_1(u_{[0,T-1]})
	\end{bmatrix} = n+m.
	\end{equation}
\end{theorem}

\vspace{8pt}

We will not provide a proof of Theorem~\ref{thmissys} at this stage, since the statement will follow from the more general Theorem~\ref{thmisosys} that will be proven in Section~\ref{secio}. Instead, we will explain the main ideas behind the result. The essence of Theorem~\ref{thmissys} is that the proposed input sequence increases the rank of the input/state Hankel matrix \emph{at every time step}. To be specific, the input sequence guarantees that 
\begin{equation}
\label{rankinc}
\rank \begin{bmatrix}
\calH_1(x_{[0,t]}) \\ \calH_1(u_{[0,t]})
\end{bmatrix} > \rank \begin{bmatrix}
\calH_1(x_{[0,t-1]}) \\ \calH_1(u_{[0,t-1]})
\end{bmatrix}
\end{equation}
for all $t = 1,2,\dots,T-1$. This increase in rank obviously occurs for any input $u(t) \in \mathbb{R}^m$ if $x(t) \not\in \im \calH_1(x_{[0,t-1]})$. Therefore, $u(t)$ can be chosen arbitrarily in this case. However, it is not evident that \eqref{rankinc} can be guaranteed if $x(t) \in \im \calH_1(x_{[0,t-1]})$. Here, the inportant intermediate step of Theorem~\ref{thmissys} is to show that there exists\footnote{These vectors can e.g. be computed using the singular value decomposition of the input/state Hankel matrix.} a $\xi$ and a nonzero $\eta$ satisfying \eqref{xieta}. The existence of these vectors heavily relies on controllability of $(A,B)$, and can be proven using a geometric argument that draws some inspiration from Hautus' proof of Heymann's lemma \cite{Hautus1977}. 

Once the existence of $\xi$ and $\eta \neq 0$ has been established, the construction of the input $u(t)$ should come without surprise. Indeed, choosing $u(t)$ such that $\xi^\top x(t) + \eta^\top u(t) \neq 0$ ensures that $\begin{bmatrix}
\xi^\top & \eta^\top 
\end{bmatrix}$ is not in the left kernel of the input/state Hankel matrix up to time $t$. By \eqref{xieta}, this implies that 
$$
\dim \leftker \begin{bmatrix}
\calH_1(x_{[0,t]}) \\ \calH_1(u_{[0,t]})
\end{bmatrix} < 
\dim \leftker \begin{bmatrix}
\calH_1(x_{[0,t-1]}) \\ \calH_1(u_{[0,t-1]})
\end{bmatrix},
$$
which is equivalent to \eqref{rankinc}. 

Note that to guarantee the rank condition \eqref{fullrank}, we require at least $T \geq n+m$ samples. A surprising fact is that Theorem~\ref{thmissys} always guarantees \eqref{fullrank} with \emph{exactly} $n+m$ samples, despite the a priori lack of knowledge of the system matrices $A$ and $B$. This makes our design method completely \emph{sample efficient} in the sense that any experiment design method requires at least as many samples as the one in Theorem~\ref{thmissys}. In particular, our method outperforms the usual condition of persistency of excitation of order $n+1$, which requires at least $T \geq nm+n+m$ samples to guarantee \eqref{fullrank}. Specifically, Theorem~\ref{thmissys} saves at least $nm$ samples compared to persistency of excitation.  The number of samples in Theorem~\ref{thmissys} is \emph{linear} (instead of quadratic) in the dimensions of the system variables.

We emphasize that the success of Theorem~\ref{thmissys} is due to its \emph{online} nature: the input $u(t)$ is computed by making use of past inputs and states, which contain valuable information about the unknown system. The method is thereby fundamentally different from the classical persistency of excitation condition, which is a purely \emph{offline} condition: one can design a persistently exciting input before collecting any data. In the language of Sontag \cite{Sontag1980}, persistently exciting inputs (of order $n+1$) are \emph{universal} in the sense that they guarantee \eqref{fullrank} \emph{for any} controllable system \eqref{systema}. In contrast, the input sequence of Theorem~\ref{thmissys} is tailored to the \emph{specific} system \eqref{systema} that has produced the past inputs and states. This allows for a reduction of the required number of data samples, at the small cost of some simple online computations. 

\begin{remark}
Additional constraints can be considered on the input $u(t)$ in Theorem~\ref{thmissys}. For example, one can incorporate the norm constraint $\norm{u(t)} = \delta$ by selecting 
$$
u(t) = \begin{cases} \frac{\delta \eta}{\norm{\eta}} & \text{ if } \xi^\top x(t) \geq 0 \\
-\frac{\delta \eta}{\norm{\eta}} & \text{ otherwise,} 
\end{cases}
$$
since this $u(t)$ satisfies $\xi^\top x(t) + \eta^\top u(t) \neq 0$. Theorem~\ref{thmissys} also demonstrates that a sequence of independent and normally distributed random inputs is sample efficient with probability 1. This is because the set $\mathcal{S} = \{u(t) \mid \xi^\top x(t) + \eta^\top u(t) = 0\}$ is a proper affine subspace of $\mathbb{R}^m$. However, we emphasize that Theorem~\ref{thmissys} deals with a general class of inputs that yields absolute (rather than probabilistic) guarantees.
\end{remark}

\subsection{Input design for input/output systems}
\label{secio}
Next, we turn our attention to input/output data generated by system \eqref{system}. In this setting, we want to find a sequence of inputs $u(0),u(1),\dots,u(T-1)$ so that the matrix \eqref{inputoutputHankel} has rank $n+mL$. As before, we opt for the \emph{online} design of these inputs, meaning that we select $u(t)$ on the basis of the collected outputs $y(0),y(1),\dots,y(t-1)$ and inputs $u(0),u(1),\dots,u(t-1)$. 

The following theorem is a building block in our approach, and shows how full row rank of the \emph{input/state} Hankel matrix 
\begin{equation}
\label{fullrankWillems}
   \begin{bmatrix} \mathcal{H}_{1}(x_{[0,T-L]}) \\ \mathcal{H}_{L}(u_{[0,T-1]}) \end{bmatrix} = 
    \begingroup 
    \setlength\arraycolsep{2pt}
    \begin{bmatrix} x(0) & x(1) & \cdots & x(T-L) \\ u(0) & u(1) & \cdots & u(T-L) \\ \vdots & \vdots &  & \vdots \\ u(L-1) & u(L) & \cdots & u(T-1) \end{bmatrix}
    \endgroup 
\end{equation}
can be guaranteed if the state of \eqref{systema} is measured.

\begin{theorem}
	\label{thmisosys}
	Consider the controllable system \eqref{systema}. Define $T := n+(m+1)L-1$ for $L \geq 1$. Choose $u(0),u(1),\dots,u(L-1) \in \mathbb{R}^m$ arbitrarily, but not all zero. Design $u(t)$ (for $t = L,L+1,\dots,T-1$) as follows:
	\begin{itemize}
		\item If 
		\begin{equation}
		\label{notinc}
		\begin{bmatrix}
		x(t-L+1) \\ u_{[t-L+1,t-1]}
		\end{bmatrix} \not\in \im \begin{bmatrix}
		\calH_1(x_{[0,t-L]}) \\ \calH_{L-1}(u_{[0,t-2]})
		\end{bmatrix},
		\end{equation}
		then choose $u(t) \in \mathbb{R}^m$ arbitrarily. 
		\item If 
		\begin{equation}
		\label{inc}
		\begin{bmatrix}
		x(t-L+1) \\ u_{[t-L+1,t-1]}
		\end{bmatrix} \in \im \begin{bmatrix}
		\calH_1(x_{[0,t-L]}) \\ \calH_{L-1}(u_{[0,t-2]})
		\end{bmatrix},
		\end{equation}
		then there exist $\xi \in \mathbb{R}^n$ and $\eta_1,\eta_2,\dots,\eta_{L} \in \mathbb{R}^m$ with $\eta_{1} \neq 0$ such that
		\begin{equation}
		\label{xietas}
		\begin{bmatrix} \xi^\top & \eta_L^\top & \cdots & \eta_2^\top & \eta_{1}^\top \end{bmatrix} \begin{bmatrix}
		\calH_1(x_{[0,t-L]}) \\ \calH_L(u_{[0,t-1]})
		\end{bmatrix} = 0. 
		\end{equation}
		In this case, choose $u(t)$ such that 
		$$
		\xi^\top x(t-L+1) + \eta_L^\top u(t-L+1) + \cdots + \eta_{1}^\top u(t) \neq 0.
		$$ 		
	\end{itemize}
	Then, the sequence $u(0),u(1),\dots,u(T-1)$ is such that 
	\begin{equation}
	\label{fullrankL}
	\rank \begin{bmatrix}
	\calH_1(x_{[0,T-L]}) \\ \calH_L(u_{[0,T-1]})
	\end{bmatrix} = n+mL.
	\end{equation}
\end{theorem}

\vspace{5pt}
\begin{proof}
	Since $u(0),u(1),\dots,u(L-1)$ are not all zero, we have
	$$
	\rank \begin{bmatrix}
	x(0) \\ u_{[0,L-1]} 
	\end{bmatrix} = 1. 
	$$
	The idea of the proof is to show that in each time step, we can increase the rank of the Hankel matrix. That is, for each $t = L,L+1,\dots,T-1$, we want to prove that
	\begin{equation}
	\label{rankincL}
	\rank \begin{bmatrix}
	\calH_1(x_{[0,t-L+1]}) \\ \calH_L(u_{[0,t]})
	\end{bmatrix} > \rank \begin{bmatrix}
	\calH_1(x_{[0,t-L]}) \\ \calH_L(u_{[0,t-1]})
	\end{bmatrix}.
	\end{equation}
	Let $t \in \{L,L+1,\dots,T-1\}$. First consider the case that \eqref{notinc} holds. Clearly, \eqref{rankincL} is satisfied for any $u(t) \in \mathbb{R}^m$. Next, consider the case that \eqref{inc} holds. We will prove by contradiction that there exist vectors $\xi$ and $\eta_1,\eta_2,\dots,\eta_L$ with $\eta_1 \neq 0$ satisfying \eqref{xietas}. If such vectors do not exist then  
	\begin{equation}
	\label{contraL}
	\leftker \begin{bmatrix}
	\calH_1(x_{[0,t-L]}) \\ \calH_L(u_{[0,t-1]})
	\end{bmatrix} =
	\leftker \begin{bmatrix} \calH_1(x_{[0,t-L]}) \\ \calH_{L-1}(u_{[0,t-2]}) \end{bmatrix} \times \{0_m\}.
	\end{equation}
	We claim that this implies that
	\begin{equation}
	\label{contraLk}
	\leftker \! \begin{bmatrix}
	\calH_1(x_{[0,t-L]}) \\ \calH_L(u_{[0,t-1]})
	\end{bmatrix} \! = \!
	\leftker \! \begin{bmatrix} \calH_1(x_{[0,t-L]}) \\ \calH_{L-i}(u_{[0,t-i-1]}) \end{bmatrix}  \times  \{0_{im}\}
	\end{equation}
	for all $i = 1,\dots,L$. By hypothesis, \eqref{contraLk} holds for $i = 1$. Suppose that \eqref{contraLk} holds for some $1 \leq i = k < L$. Our goal is to show that \eqref{contraLk} holds for $i = k+1$ as well. 
	Let 
	\begin{equation}
	\label{leftker}
	\begin{bmatrix}
	\xi^\top & \eta_L^\top & \cdots & \eta_{1}^\top
	\end{bmatrix} \begin{bmatrix}
	\calH_1(x_{[0,t-L]}) \\ \calH_L(u_{[0,t-1]})
	\end{bmatrix} = 0,
	\end{equation}
	where $\xi \in \mathbb{R}^n$ and $\eta_1,\dots,\eta_{L} \in \mathbb{R}^m$. By the induction hypothesis, we have $\eta_{1} = \cdots = \eta_k = 0$ so that
	\begin{equation}
	\label{kernelk+1}
	\begin{bmatrix}
	\xi^\top & \eta_L^\top & \cdots & \eta_{k+1}^\top
	\end{bmatrix} \begin{bmatrix}
	\calH_1(x_{[0,t-L]}) \\ \calH_{L-k}(u_{[0,t-k-1]})
	\end{bmatrix} = 0.
	\end{equation} 
	By \eqref{inc}, we obtain 
	$$
	\begin{bmatrix}
		x(t-L+1) \\ u_{[t-L+1,t-k]}
		\end{bmatrix} \in \im \begin{bmatrix}
		\calH_1(x_{[0,t-L]}) \\ \calH_{L-k}(u_{[0,t-k-1]})
		\end{bmatrix}.
	$$
	Combining the latter inclusion with \eqref{kernelk+1} yields 
	\begin{equation}
	\label{kernelk+12}
	\begin{bmatrix}
	\xi^\top & \eta_L^\top & \cdots & \eta_{k+1}^\top
	\end{bmatrix} \begin{bmatrix}
	\calH_1(x_{[0,t-L+1]}) \\ \calH_{L-k}(u_{[0,t-k]})
	\end{bmatrix} = 0.
	\end{equation}
	Note that the laws of system \eqref{systema} imply that
	$$
	\begin{bmatrix}
	\calH_1(x_{[1,t-L+1]}) \\ \calH_{L-k}(u_{[1,t-k]})
	\end{bmatrix} = 
	\begin{bmatrix}
	A & B & 0 \\ 0 & 0 & I
	\end{bmatrix}
	\begin{bmatrix}
	\calH_1(x_{[0,t-L]}) \\ \calH_{L-k+1}(u_{[0,t-k]})
	\end{bmatrix}.
	$$
	Therefore, \eqref{kernelk+12} implies 
	$$
	\begin{bmatrix}
	\xi^\top A & \xi^\top B & \eta_L^\top & \cdots & \eta_{k+1}^\top
	\end{bmatrix} \begin{bmatrix}
	\calH_1(x_{[0,t-L]}) \\ \calH_{L-k+1}(u_{[0,t-k]})
	\end{bmatrix} = 0.
	$$
	This yields 
	$$
	\begin{bmatrix}
	\xi^\top A \! & \xi^\top B \! & \eta_L^\top & \! \cdots \! & \eta_{k+1}^\top \! & 0_{(k-1)m}^\top 
	\end{bmatrix} \!\! \begin{bmatrix}
	\calH_1(x_{[0,t-L]}) \\ \calH_{L}(u_{[0,t-1]})
	\end{bmatrix} \! = 0.
	$$
	Finally, by the induction hypothesis we see that $\eta_{k+1} = 0$, as desired. Thus, \eqref{contraLk} holds for all $i = 1,\dots,L$. In particular, for $i = L$ we obtain  
	$$
	\leftker \begin{bmatrix}
	\calH_1(x_{[0,t-L]}) \\ \calH_L(u_{[0,t-1]})
	\end{bmatrix} = \leftker \calH_1(x_{[0,t-L]}) \times \{0_{mL}\},
	$$  
	equivalently (using the fact that $(\im X)^\perp = \leftker X$), 
	\begin{equation}
	\label{imagemL}
	\im \begin{bmatrix}
	\calH_1(x_{[0,t-L]}) \\ \calH_L(u_{[0,t-1]})
	\end{bmatrix} = \im \calH_1(x_{[0,t-L]}) \times \mathbb{R}^{mL}.
	\end{equation}
	Multiplication on both sides by the matrix $\begin{bmatrix} A & B & 0 \end{bmatrix}$ yields
	$$
	A \im \calH_1(x_{[0,t-L]}) + \im B = \im \calH_1(x_{[1,t-L+1]}).
	$$
	Using \eqref{inc}, we obtain 
	$$
	A \im \calH_1(x_{[0,t-L]}) + \im B \subseteq \im \calH_1(x_{[0,t-L]}).
	$$
	In particular, we see that 
	\begin{align*}
	A \im \calH_1(x_{[0,t-L]}) &\subseteq \im \calH_1(x_{[0,t-L]}) \\ 
	\im B &\subseteq \im \calH_1(x_{[0,t-L]}).
	\end{align*}
	In other words, $\im \calH_1(x_{[0,t-L]})$ is an $A$-invariant subspace containing $\im B$. Since the reachable subspace $\langle A \mid \im B\rangle$ of the pair $(A,B)$ is the \emph{smallest} $A$-invariant subspace containing $\im B$ (cf. \cite[Ch. 3]{Trentelman2001}), we see that
	$$
	\mathbb{R}^n = \langle A \mid \im B\rangle \subseteq \im \calH_1(x_{[0,t-L]}),
	$$
	where we made use of the fact that $(A,B)$ is controllable. As such, $\im \calH_1(x_{[0,t-L]}) = \mathbb{R}^n$ and by \eqref{imagemL} we see that
	$$
	\im \begin{bmatrix}
	\calH_1(x_{[0,t-L]}) \\ \calH_L(u_{[0,t-1]})
	\end{bmatrix} = \mathbb{R}^{n+mL}.
	$$
	This implies that $t \geq n+(m+1)L-1$ which leads to a contradiction as $t \in \{L,L+1,\dots,T-1\}$. As such, we conclude that \eqref{contraL} does not hold. Therefore, there exist $\xi \in \mathbb{R}^n$ and $\eta_1,\eta_2,\dots,\eta_{L} \in \mathbb{R}^m$ with $\eta_{1} \neq 0$ such that \eqref{xietas} holds. Clearly, this means that there exists a $u(t)$ such that 
	$$
	\xi^\top x(t-L+1) + \eta_L^\top u(t-L+1) + \cdots + \eta_{1}^\top u(t) \neq 0.
	$$ 		 
	For such an input, we have that
	$$
	\dim \leftker \begin{bmatrix}
	\calH_1(x_{[0,t-L]}) \\ \calH_L(u_{[0,t-1]})
	\end{bmatrix} >
	\dim \leftker \begin{bmatrix}
	\calH_1(x_{[0,t-L+1]}) \\ \calH_L(u_{[0,t]})
	\end{bmatrix},
	$$
	and consequently, \eqref{rankincL} holds. This proves the theorem. 
\end{proof}

Note that in the special case $L = 1$, Theorem~\ref{thmisosys} reduces to Theorem~\ref{thmissys}. Next, we turn our attention to the situation in which we measure inputs and outputs and want to ensure that \eqref{inputoutputHankel} has rank $n + mL$. This is the topic of the next theorem. 

\begin{theorem}
	\label{thmiosys}
	Consider the minimal system \eqref{system}. Let $L\! >\! \ell$ and $T \!:=\! n+(m+1)L-1$. Select $u(0),u(1),\dots,u(L-1) \in \mathbb{R}^m$ arbitrarily, but not all zero. Furthermore, design $u(t)$ (for $t = L,L+1,\dots,T-1$) as follows:
	\begin{itemize}
		\item If 
		\begin{equation}
		\label{notincio}
		\begin{bmatrix}
		y_{[t-L+1,t-1]} \\ u_{[t-L+1,t-1]}
		\end{bmatrix} \not\in \im \begin{bmatrix}
		\calH_{L-1}(y_{[0,t-2]}) \\ \calH_{L-1}(u_{[0,t-2]})
		\end{bmatrix},
		\end{equation}
		then choose $u(t) \in \mathbb{R}^m$ arbitrarily. 
		\item If 
		\begin{equation}
		\label{incio}
		\begin{bmatrix}
		y_{[t-L+1,t-1]} \\ u_{[t-L+1,t-1]}
		\end{bmatrix} \in \im \begin{bmatrix}
		\calH_{L-1}(y_{[0,t-2]}) \\ \calH_{L-1}(u_{[0,t-2]})
		\end{bmatrix},
		\end{equation}
		there exist $\xi_1,\dots,\xi_{L-1} \in \mathbb{R}^p$ and $\eta_1,\dots,\eta_{L} \in \mathbb{R}^m$ with $\eta_{1} \neq 0$ such that
		\begin{equation}
		\label{xietasio}
		\begin{bmatrix} \xi_{L-1}^\top & \! \cdots \! & \xi_1^\top & \eta_L^\top & \! \cdots \! & \eta_{1}^\top \end{bmatrix} \begin{bmatrix}
		\calH_{L-1}(y_{[0,t-2]}) \\ \calH_L(u_{[0,t-1]})
		\end{bmatrix} = 0. 
		\end{equation}
		In this case, choose $u(t)$ such that
		\begin{equation}
		\label{xietaio}
		\begin{aligned}
		\xi_{L-1}^\top &y(t-L+1) + \cdots + \xi_1^\top y(t-1) \\ + \eta_L^\top &u(t-L+1) + \cdots + \eta_{1}^\top u(t) \neq 0.
		\end{aligned}
		\end{equation}
			\end{itemize} 		
	Then we have that
	\begin{equation}
	\label{rankLio}
	\rank \begin{bmatrix}
	\calH_L(y_{[0,T-1]}) \\ \calH_L(u_{[0,T-1]})
	\end{bmatrix} = n+mL.
	\end{equation}
\end{theorem}
\vspace{10pt}

\begin{proof}
	Since $u(0),u(1),\dots,u(L-1)$ are not all zero, we have 
	$$
	\rank \begin{bmatrix}
	y_{[0,L-2]} \\ u_{[0,L-1]} 
	\end{bmatrix} = 1. 
	$$
	The idea of the proof is to show that for each time step $t = L,L+1,\dots,T-1$, we have
	\begin{equation}
	\label{rankincLio}
	\rank \begin{bmatrix}
	\calH_{L-1}(y_{[0,t-1]}) \\ \calH_L(u_{[0,t]})
	\end{bmatrix} > \rank \begin{bmatrix}
	\calH_{L-1}(y_{[0,t-2]}) \\ \calH_L(u_{[0,t-1]})
	\end{bmatrix}.
	\end{equation}
	This would prove that 
	$$
	\rank \begin{bmatrix}
	\calH_{L-1}(y_{[0,T-2]}) \\ \calH_L(u_{[0,T-1]})
	\end{bmatrix} = n+mL,
	$$
	and consequently, \eqref{rankLio} holds. 
	
	Let $t \in \{L,L+1,\dots,T-1\}$. First consider the case that \eqref{notincio} holds. Clearly, \eqref{rankincLio} is satisfied for any $u(t) \in \mathbb{R}^m$.
	
	Next, consider the case that \eqref{incio} holds. We claim that 
	\begin{equation}
	\label{claimLio}
	\leftker \! \begin{bmatrix}
	\calH_{L-1}(y_{[0,t-2]}) \\ \calH_L(u_{[0,t-1]})
	\end{bmatrix} \neq \leftker \! \begin{bmatrix} \calH_{L-1}(y_{[0,t-2]}) \\ \calH_{L-1}(u_{[0,t-2]}) \end{bmatrix} \times \{0_m\}.
	\end{equation}
	We will prove this claim by contradiction. Thus, suppose that \eqref{claimLio} does hold with equality, equivalently,
	\begin{equation}
	\label{imeqio}
	\im \begin{bmatrix}
	\calH_{L-1}(y_{[0,t-2]}) \\ \calH_L(u_{[0,t-1]})
	\end{bmatrix} = \im \begin{bmatrix} \calH_{L-1}(y_{[0,t-2]}) \\ \calH_{L-1}(u_{[0,t-2]}) \end{bmatrix} \times \mathbb{R}^m.
	\end{equation} 
Define the matrices $N$ and $M$ as
	$$
	N := \begin{bmatrix}
	\mathcal{O}_{L-1} & \mathcal{T}_{L-1} \\
	0 & I
	\end{bmatrix} \text{ and }
	M := \begin{bmatrix}
	N & 0 \\
	0 & I_m
	\end{bmatrix},
	$$
	where the observability matrix $\mathcal{O}_i$ is defined in \eqref{observabilitymatrix} and the Toeplitz matrix $\mathcal{T}_{i}$ is given by 
	\begin{align*}
	\mathcal{T}_i := \begin{bmatrix}
	D & 0 & 0 &  \cdots & 0 \\
	CB & D &  0 & \cdots & 0 \\
	CAB & CB & D & \cdots & 0 \\
	\vdots & \vdots & \vdots & \ddots & \vdots \\
	CA^{i-2}B & CA^{i-3}B & CA^{i-4}B & \cdots & D 
	\end{bmatrix}.
	\end{align*}	
	Then \eqref{imeqio} implies 
	$$
	M \im \begin{bmatrix}
	\calH_1(x_{[0,t-L]}) \\ \calH_L(u_{[0,t-1]})
	\end{bmatrix} = M \left( \im \begin{bmatrix}
	\calH_1(x_{[0,t-L]}) \\ \calH_{L-1}(u_{[0,t-2]})
	\end{bmatrix} \times \mathbb{R}^m \right).
	$$
	Since $(C,A)$ is observable and $L > \ell$, the matrix $M$ has full column rank. As such, we conclude that 
	$$
	\im \begin{bmatrix}
	\calH_1(x_{[0,t-L]}) \\ \calH_L(u_{[0,t-1]})
	\end{bmatrix} = \im \begin{bmatrix}
	\calH_1(x_{[0,t-L]}) \\ \calH_{L-1}(u_{[0,t-2]})
	\end{bmatrix} \times \mathbb{R}^m,
	$$
	i.e., \eqref{contraL} holds. Similarly, we note that \eqref{incio} implies
	$$
	N \begin{bmatrix}
	x(t-L+1) \\ u_{[t-L+1,t-1]}
\end{bmatrix} \in N \im \begin{bmatrix}
	\mathcal{H}_1(x_{[0,t-L]} \\
	\mathcal{H}_{L-1}(u_{[0,t-2]})
	\end{bmatrix}.
	$$	
	Since $N$ has full column rank, this implies that \eqref{inc} holds. It now follows from the proof of Theorem~\ref{thmisosys} that $t$ satisfies $t \geq n+(m+1)L-1$. This yields a contradiction since $t \in \{L,L+1,\dots,T-1\}$. As such, \eqref{claimLio} holds. Therefore, there exist $\xi_1,\xi_2,\dots,\xi_{L-1} \in \mathbb{R}^p$ and $\eta_1,\eta_2,\dots,\eta_{L} \in \mathbb{R}^m$ with $\eta_{1} \neq 0$ such that \eqref{xietasio} is satisfied. This means that we can choose an input $u(t)$ such that \eqref{xietaio} holds. 
	For this choice of $u(t)$, we obtain 
	$$
	\dim \leftker \! \begin{bmatrix}
	\calH_{L-1}(y_{[0,t-2]}) \\ \calH_L(u_{[0,t-1]})
	\end{bmatrix} \!>\!
	\dim \leftker \! \begin{bmatrix}
	\calH_{L-1}(y_{[0,t-1]}) \\ \calH_L(u_{[0,t]})
	\end{bmatrix}.
	$$	
	We conclude that \eqref{rankincLio} holds, which proves the theorem.
\end{proof}

We note that \eqref{rankLio} can only hold if $T \geq n+(m+1)L-1$. By Theorem \ref{thmiosys}, we can always design an informative experiment of length \emph{exactly} $n+(m+1)L-1$. In other words, the design procedure of Theorem \ref{thmiosys} is again sample efficient. In comparison, the usual persistency of excitation condition of order $n+L$ (Proposition~\ref{theoremWillems}) requires $T \geq n+(m+1)L+mn-1$ samples. Theorem~\ref{thmiosys} thus saves at least $mn$ samples when compared to persistency of excitation.

Finally, is it noteworthy that exact knowledge of the state-space dimension is \emph{not required} for Theorem~\ref{thmiosys}, even though the time horizon $T$ is defined in terms of $n$. The only requirement is an upper bound $L > \ell$ on the lag of the system. The reason is that all the steps of Theorem~\ref{thmiosys} can be executed without knowledge of $n$. One should only take some care with the stopping criterion since $T$ is a priori unknown. A simple approach to deal with this is the following: apply the steps of Theorem~\ref{thmiosys} for $t = 0,1,\dots$ until $t$ is such that \eqref{imeqio} is satisfied. By the proof of Theorem~\ref{thmiosys}, the smallest $t$ for which \eqref{imeqio} holds is $t = n+(m+1)L-1$. From this relation, $n$ can be recovered, and for $T := t$, the rank condition \eqref{rankLio} holds. 

\section{Conclusions and discussion} 
\label{sec:conc}
The purpose of this paper has been to provide a new angle of attack to the problem of input design for data-driven modeling and control of linear systems. Instead of the classical persistency of excitation condition, we have proposed an \emph{online} method to select the inputs. We have studied both input/state and input/output systems. For both types of systems, we have provided an input design method that guarantees an important rank condition on the data Hankel matrix. A notable feature of our approach is that it is completely \emph{sample efficient} in the sense that it guarantees these rank conditions with the minimum number of samples.

The presented results apply to exact data, and the generalization to noisy data is arguably the most important problem for future work. In the case of noisy measurements, rank conditions on Hankel matrices are not sufficient to perform succesful data-driven modeling and control. Nonetheless, under suitable conditions, controllers with guaranteed stability and performance can be obtained from data. For this, data-based linear matrix inequality (LMI) conditions were studied in \cite{DePersis2020} and \cite{Berberich2020}, and a necessary and sufficient condition was provided in \cite{vanWaarde2020d} using a generalization of Yakubovich's S-lemma \cite{Yakubovich1977}.

Future work could thus focus on experiment design techniques that aim at making such data-based LMI's feasible. One possible avenue to explore is to apply the repeated dimension reduction argument in this paper to other sets. In the current work, we have shown that it is possible to reduce the dimension of the \emph{left kernel} of a Hankel matrix at every time step. However, in the above context of data-driven control, a fundamental role is played by the solution set to the \emph{dual LMI} \cite{Balakrishnan2003} of the data-based LMI's for control. This dual LMI is infeasible if and only if its primal LMI is feasible. Thus, a systematic reduction of the dimension of the dual solution set would guarantee feasibility of the data-based LMI's in a finite number of steps. Constructing inputs that achieve this is a main challenge for future research.

\bibliographystyle{IEEEtran}
\bibliography{references}

\end{document}